\documentclass[11pt,a4paper,reqno]{amsart}
\usepackage[all]{xy}
\usepackage{amsmath,amsfonts,amssymb,amscd}
\usepackage{soul,color}
\theoremstyle{plain}
\newtheorem{theorem}{Theorem}[section]
\newtheorem{mytheorem}{Theorem}[subsection]
\newtheorem{corollary}[mytheorem]{Corollary}
\newtheorem{lemma}[mytheorem]{Lemma}
\newtheorem{proposition}[mytheorem]{Proposition}
\newtheorem{definition}[mytheorem]{Definition}
\newtheorem{remark}[mytheorem]{Remark}

\font\cyr wncyr10 at 11pt \def\Sha{\hbox{\cyr X}}

\def\tF{\textsc{F}}

\def\AA{{\mathbb A}}
\def\CC{{\mathbb C}}

\def\FF{{\mathbb F}}

\def\QQ{{\mathbb Q}}
\def\RR{{\mathbb R}}
\def\ZZ{{\mathbb Z}}

\def\A{{\mathcal A}}

\def\cL{{\mathcal L}}
\def\M{{\mathcal M}}

\def\cO{{\mathcal O}}

\def\cS{{\mathcal S}}

\newcommand{\fL}{\mathfrak{L}}

\newcommand{\fm}{\mathfrak{m}}

\newcommand{\kpinf}{{k}_{\infty}^{(p)}}
\newcommand{\ilim}{\varprojlim_n}
\newcommand{\Kpinf}{{K}_{\infty}^{(p)}}
\newcommand{\cal}{\mathcal}
\newcommand{\La}{\operatorname{\Lambda}}
\newcommand{\lra}{\longrightarrow}

\newcommand{\lr}{{\longrightarrow\;}}
\newcommand{\mbb}{\mathbb}

\newcommand{\dlim}{\varinjlim_n}

\def\+{{\dagger}}
\DeclareMathOperator{\Ker}{Ker}
\DeclareMathOperator{\Coker}{Coker}

\DeclareMathOperator{\coh}{H}

\DeclareMathOperator{\Fr}{Fr}

\DeclareMathOperator{\Gal}{Gal}

\DeclareMathOperator{\image}{Im}
\DeclareMathOperator{\Sel}{Sel}

\DeclareMathOperator{\Hom}{Hom}

\DeclareMathOperator{\rank}{rank}

\definecolor{purple}{rgb}{0.7,0,1}

\begin{document}
\title[ ] {The Iwasawa Main conjecture for semistable abelian varieties over function fields }

\author[Lai] {King Fai Lai}
\address{School of Mathematical Sciences\\
Capital Normal University\\
Beijing 100048, China}
\email{kinglaihonkon@gmail.com}

\author[Longhi] {Ignazio Longhi}
\address{Department of Mathematical Sciences\\
Xi'an Jiaotong-Liverpool University\\
Suzhou Industrial Park, Suzhou  215123, China}
\email{Ignazio.Longhi@xjtlu.edu.cn}

\author[Tan]{Ki-Seng Tan}
\address{Department of Mathematics\\
National Taiwan University\\
Taipei 10764, Taiwan}
\email{tan@math.ntu.edu.tw}

\author[Trihan]{Fabien Trihan}
\address{College of Engineering, Mathematics and Physical Sciences\\
University of Exeter\\
North Park Road, Exeter, UK}
\email{f.trihan@exeter.ac.uk}



\begin{abstract} We prove the Iwasawa Main Conjecture over the arithmetic $\ZZ_p$-extension for semistable abelian varieties over function fields of characteristic $p>0$.
\end{abstract}

\maketitle

\tableofcontents

\section{Introduction}  \label{sec:intro}

We prove in this paper an important case of the Iwasawa Main Conjecture for abelian varieties over
function fields. We fix a prime number $p$ ($p=2$ is allowed). Let $K$ be a global field of characteristic $p$.
Let $\Kpinf$ be the unramified $\ZZ_p$-extension of $K$ (called the arithmetic extension of $K$), with $\Gamma:=\Gal(\Kpinf/K)$,
and write
$Q(\Lambda)$ for the fraction field of
 the Iwasawa algebra $\La :=\ZZ_p[[\Gamma]]$.
 Let $A$ be an abelian variety over $K$ with semistable reduction.
Let $X_p(A/\Kpinf)$ denote the Pontryagin dual of the Selmer group $\Sel_{p^\infty}(A/\Kpinf)$. Then $X_p(A/\Kpinf)$ is finitely generated over $\La$, hence we can define the characteristic ideal $\chi\big(X_p(A/\Kpinf)\big)$.
It is a principal ideal of $\Lambda$, and we let $c_{A/\Kpinf}\in\Lambda$ denote a generator, which is  unique up to elements in $\Lambda^{\times}$.

For $\omega$ a continuous character of $\Gamma$ and $T$ a finite set of places of $K$, let $L_T(A,\omega,s)$ denote the twisted Hasse-Weil $L$-function of $A$ with the local factors at $T$ taken away. In practise $T$ will consist of those places where $A$ has bad reduction.
Our  main theorem is the following:

\begin{theorem} \label{t:summary} There exists a ``$p$-adic $L$-function'' $\cL_{A/\Kpinf}\in Q(\La)$ such that for any continuous character
$\omega\colon\Gamma\rightarrow\CC^\times$, $\omega(\cL_{A/\Kpinf})$ is defined and
\begin{equation}\label{e:imc2interpolate}
\omega(\cL_{A/\Kpinf})= L_T(A,\omega,1).
\end{equation}
Furthermore,
\begin{equation}\label{e:imc3characteristic}
\mathcal{L}_{A/\Kpinf}\equiv \star_{A,\Kpinf}\cdot c_{A/\Kpinf}\mod\La^\times.
\end{equation}
\end{theorem}

The interpolation formula \eqref{e:imc2interpolate} is proven in Theorems \ref{GIMC2}, while \eqref{e:imc3characteristic} is the content of Theorems \ref{GIMC3}. For the precise expression of $\star_{A,\Kpinf}$ we refer to Proposition \ref{p:summarize}.\\

\subsubsection{A closer look at our result}
In order to explain our results, we need first to introduce some cohomology groups and operators between them.

We write $C/\FF$ for  the smooth proper geometrically connected curve  which is the model of the function field $K$ over its field of constants $\FF$. Let $C_\infty:=C\times_{\FF} \kpinf$ (where $\kpinf$ denotes the $\ZZ_p$-extension of $\FF$) and $\pi\colon C_\infty\to C$ be the pro-\'etale covering with Galois group $\Gal(\Kpinf/K)$.

Let $\cal A$ denote the N\'eron model of $A$ over $C$. Let  $Z$ be the finite  set of points where $A$ has bad reduction. Denote by  $Lie(\A)$ the Lie algebra of $\A$. Let $L^i_\infty$ be the $i$th cohomology group of
$$\RR\Gamma\big(C_{\infty},\pi^*Lie(\A)(-Z))\big)\otimes^{\mbb{L}}\QQ_p/\ZZ_p.$$
Let $D$ be the covariant log Dieudonn\'e crystal  associated with $A/K$ as constructed in \cite[IV]{KT03}.
The syntomic complex ${\cal S}_D$ is the mapping fibre of ``$1-$Frobenius''  in the derived category of complexes of sheaves over $C_{\acute{e}t}$ (\cite[\S 5.8]{KT03}). Let $N^i_\infty$ be the $i$th cohomology group of 
$$\RR\Gamma(C_\infty,\pi^*\cS_D)\otimes^{\mbb{L}}\QQ_p/\ZZ_p.$$

\begin{theorem} For $i=0,1,2$ and $j=0,1$, $(N^i_\infty)^\vee$ and $(L^j_\infty)^\vee$ are finitely generated torsion $\La$-modules.
\end{theorem}

\noindent Here ${}^\vee$ denotes the Pontryagin dual. The proof shall be given in Corollaries \ref{M,L} and \ref{1case}. Note that, by \cite{KT03}, $X_p(A/\Kpinf)$ as a submodule of $(N^1_\infty)^\vee$ is also $\La$-torsion.

Formula \eqref{e:charelement} in section \ref{su:charelem} will define $f_{A/\Kpinf}\in Q(\La)/\La^\times$ as the alternating product of the characteristic elements associated with the $(N^i_\infty)^\vee$ and $(L^j_\infty)^\vee$.
Because of the relation between $X_p(A/\Kpinf)$ and $(N^1_\infty)^\vee$, we can write
$$f_{A/K_{\infty}^{(p)}}=\star_{A,K_{\infty}^{(p)}}\cdot c_{A/K_{\infty}^{(p)}}\,$$
where $\star_{A,K_{\infty}^{(p)}}$ consist of terms whose arithmetic meaning is explained in Proposition \ref{p:summarize}.

We define the $p$-adic $L$-function $\cL_{A/\Kpinf}$ as the alternating product of determinants of the action of ``$1-$Frobenius'' on the log crystalline cohomology of $D(-Z)$ (see Section \ref{su:padicL} for the precise expression). Theorem \ref{GIMC2} proves that $\cL_{A/\Kpinf}$ satisfies the interpolation formula \eqref{e:imc2interpolate}, with $T=Z$.

Finally, we can state our analogue of the Iwasawa Main Conjecture:

\begin{theorem}\label{GIMC3} Let $A$ be an
abelian variety with at worst semistable reduction relative to the arithmetic extension $\Kpinf/K$.
We have the following equality in $Q(\La)^\times/\La^\times${\em :}
$$\cL_{A/\Kpinf}=f_{A/\Kpinf}$$
\end{theorem}

\noindent The proof is based on a generalization of a lemma of $\sigma$-linear algebra that was used to prove the cohomological formula of the Birch and Swinnerton-Dyer conjecture (see \cite[Lemma 3.6]{KT03}).

Section \ref{su:Eulerchar} investigates the consequences of Theorem \ref{GIMC3} in the direction of a $p$-adic Birch and Swinnerton-Dyer conjecture. The following result can be seen as a geometric analogue of the conjecture of Mazur-Tate-Teitelbaum (\cite{mtt86}):

\begin{theorem} Assume that $A/K$ has semistable reduction. Then
$$ord\big(\cL_{A/\Kpinf}\big)=ord_{s=1}\big(L_Z(A,s)\big)\geq\rank_\ZZ A(K)\,.$$
If moreover $A/K$ verifies the Birch and Swinnerton-Dyer Conjecture, the inequality above becomes an equality and
$$|L(\cL_{A/\Kpinf})|_p^{-1}\equiv  c_{BSD}\cdot |(N_\infty^2)_\Gamma| \mod \ZZ_p^\times \,,$$
where $c_{BSD}$ is the leading coefficient at $s=1$ of $L_Z(A,s)$.
\end{theorem}

\noindent Here $ord$ denotes the ``analytic rank'' and $L$ the ``leading coefficient'' of power series in $\La$ (see \S\ref{ss:twistEuler} for precise definitions). The proof will be provided in Theorem \ref{euchar}. Proposition \ref{p:N2trivial} will prove that very often the error term $|(N_\infty^2)_\Gamma|$ is just 1.

It seems difficult at the present to remove the semistable hypothesis and will require a delicate study of the integral $p$-adic cohomology computed over (wildly) ramified extensions.

Finally we would like to mention to the readers \cite{LLTT1} for a proof of the Iwasawa Main conjecture for constant ordinary abelian varieties over $\ZZ_p^d$- extensions ramifying at a finite set of places (the proof is using \cite{LLTT2} giving a functional equation of the Selmer groups of abelian varieties over general global fields) and \cite{LLTT3}, for a case of non-torsion Selmer group where the Main conjecture  nevertheless holds.


\begin{subsubsection}*{Acknowledgements} The fourth author has been supported by EPSRC. He would like also to express his gratitude to Takeshi Saito for his hospitality at the University of Tokyo where part of this work has been written. Authors 2, 3 and 4 thank Centre de Recerca Matem\`atica for hospitality while working on part of this paper.
Authors 1, 2 and 3 have been partially supported by the National Science Council of Taiwan, grants NSC98-2115-M-110-008-MY2, NSC100-2811-M-002-079 and NSC99-2115-M-002-002-MY3 respectively. Finally, it is our pleasure to thank NCTS/TPE for supporting a number of meetings of the authors in National Taiwan University. \end{subsubsection}

\section{Syntomic cohomology of abelian varieties} \label{s:genord}

\subsubsection{The arithmetic tower} We fix the notations. For any $n\geq 0$, let $k_n/\FF$ be the
$\ZZ/p^n\ZZ$-extension of $\FF$ and let $\kpinf:=\cup_{n\geq 0} k_n$ denote the induced $\ZZ_p$-extension of $\FF$. Thus our tower becomes $K_n:=Kk_n$ and $\Kpinf:=K\kpinf$.
Since they are canonically isomorphic, we identify $\Gamma$, $\Gamma_n$ and $\Gamma^{(n)}$ with $\Gal(\kpinf/\FF)$, $\Gal(k_n/\FF)$ and $\Gal(\kpinf/k_n)$ respectively. We denote by $\Fr_q$ the generator of $\Gal(\kpinf/\FF)$, $x\mapsto x^q$.

Let $C/\FF$ denote the smooth proper geometrically connected curve which is the model of $K$ over $\FF$. Let $C_\infty:=C\times_{\FF} \kpinf$ and $C_n:=C\times_{\FF}k_n$. Let $\pi\colon C_\infty\to C$ and $\pi_n\colon C_n\to C$ denote the (pro) \'etale covering with Galois group $\Gamma$ and $\Gamma_n$ respectively. By abuse of notation, we will also denote by $\pi$ and  $\pi_n$ the associated morphisms in the log crystalline topos (\cite{bbm82}).

\subsection{The cohomology} \label{su:cohom}

\subsubsection{The Dieudonn\'e crystal} Let $\cal A$ denote the N\'eron model of $A$ over $C$. Let $U$ be the dense open subset of $C$ where $A$ has good reduction and $Z:=C- U$ the finite (possibly empty) set of points where $A$ has bad (at worst semistable) reduction. We endow $C$ with the log structure induced by the smooth divisor $Z$ and denote $C^\#$ this log-scheme. Let $D$ be the (covariant) log Dieudonn\'e crystal over $C^\#/W(\FF)$ associated with $A/K$ as constructed in \cite[IV]{KT03}. Recall the
following theorem of \cite{KT03}:

\begin{mytheorem} \label{Lie(D)} {\em(}\cite[\S 5.4(b) and \S 5.5]{KT03}{\em)} Let $i$ be the canonical morphism of topoi of \cite{bbm82} from the topos of sheaves on $C_{{\acute{e}t}}$ to the log crystalline topos $\left(C^\#/W(\FF)\right)_{crys}$.
There exists a surjective map of sheaves $D\to i_*(Lie({\cal A}))$ in $\left(C^\#/W(\FF)\right)_{crys}$.
\end{mytheorem}


\subsubsection{A distinguished triangle} We denote by $D^0$ the kernel of $D\rightarrow i_{\ast}(Lie({\cal A}))$ in the topos $(C^\#/W(\FF))_{crys}$. Let ${\bf 1}\colon D^0\rightarrow D$ be the natural inclusion. By applying the canonical projection $u_{\ast}$ from the log crystalline topos $(C^\#/W(\FF))_{crys}$ to the topos of sheaves on $C_{\acute{e}t}$, we get a distinguished triangle:
$$Ru_{\ast} D^0 \overset{\bf 1}{\lra}  Ru_{\ast} D\lra Lie({\cal A}).$$
We can twist this triangle by the divisor $Z$ to get a triangle:
\begin{equation}\label{T1}
Ru_{\ast} D^0 (-Z) \overset{\bf 1}{\lra}  Ru_{\ast} D (-Z) \lra Lie({\cal A})(-Z).
\end{equation}
where $D(-Z)$ is the twist of the log Dieudonn\'e crystal $D$ defined in \cite[\S 5.11]{KT03}.

\subsubsection{The syntomic complex} In \cite[\S 5.8]{KT03}, a Frobenius operator
$$\varphi\colon Ru_*D^0(-Z)\lra Ru_{\ast}D(-Z)$$
is constructed. We denote by ${\cS}_D$ the mapping fiber of the map
$${\bf 1}-\varphi\colon Ru_*D^0(-Z) \lra Ru_*D(-Z).$$
This complex is an object in the derived category of complexes of sheaves over $C_{\acute{e}t}$ and we have a distinguished triangle:
\begin{equation}\label{T2} \begin{CD}
\cS_D \lra Ru_{\ast} D^0(-Z) @>{{\bf 1}-\varphi}>> Ru_{\ast}D(-Z). \end{CD}
\end{equation}

\subsubsection{The cohomology theories} \label{ss:cohomtheories}


%


We define the following modules:
\begin{enumerate}
\item Let
$$P^i_n:=\coh^i_{\mathrm{crys}}(C_n^\#/W(k_n),\pi_n^*D(-Z))\,.$$
Then for any $n$, $P^i_n$ is a finitely generated $W(k_n)$-module endowed with a $\Fr_q$-linear operator $F_{i,n}$ induced by the Frobenius operator of the Dieudonn\'e crystal. Using the (log) crystalline base change by the morphism of topoi $\pi_n\colon(C_\infty^\#/W(k_n))_{crys}\to (C^\#/W(\FF))_{crys}$
(\cite[\S2.5.2]{Ka94}) and by flatness of the extensions $W(k_n)/\ZZ_p$, we have,
for $n\geq 1$,
\begin{equation} \label{e:pi0n} P^i_n\simeq P^i_0\otimes W(k_n)\,. \end{equation}
These isomorphisms identify the $\Fr_q$-linear operator $F_{i,n}$ on the left hand side with the $\Fr_q$-linear operator $F_{i,0}\otimes\Fr_q$ on the right hand side.

\item Let $M_{1,\infty}^i$ be the $i$th cohomology group of $$\RR\Gamma_{crys}(C_\infty^\#/W(\kpinf),\pi^*D^0(-Z))\otimes^{\mbb{L}} \QQ_p/\ZZ_p.$$
\item Let $M_{2,\infty}^i$ be the $i$th cohomology group of $$\RR\Gamma_{crys}(C_\infty^\#/W(\kpinf),\pi^*D(-Z))\otimes^{\mbb{L}} \QQ_p/\ZZ_p.$$
\item Let $M_{1,n}^i$ be the $i$th cohomology group of $$\RR\Gamma_{crys}(C_n^\#/W(k_n),\pi_n^*D^0(-Z))\otimes^{\mbb{L}} \QQ_p/\ZZ_p.$$
\item Let $M_{2,n}^i$ be the $i$th cohomology group of $$\RR\Gamma_{crys}(C_n^\#/W(k_n),\pi^*_nD(-Z))\otimes^{\mbb{L}} \QQ_p/\ZZ_p.$$

Again by the base change theorem, we have, for $k=1,2$, an isomorphism of torsion $W(\kpinf)$-modules
\begin{equation} \label{e:mi0infty} M^i_{k,\infty}\simeq M^i_{k,0}\otimes W(\kpinf) \end{equation}
and for any $n\geq 0$ an isomorphism of torsion $W(k_n)$-modules $$M^i_{k,n}\simeq M^i_{k,0}\otimes W(k_n)$$ identifying the $\Fr_q$-linear operator ${\bf 1}-\varphi_{i,n}$ on the left
hand side with the $\Fr_q$-linear operator ${\bf 1}\otimes id-\varphi_{i,0}\otimes\Fr_q$ on the right hand side.
\item Let $L^i_\infty$ be the $i$th cohomology group of
$$\RR\Gamma\big(C_{\infty},\pi^*Lie(\A(-Z))\big)\otimes^{\mbb{L}}\QQ_p/\ZZ_p
=\RR\Gamma\big(C_\infty,\pi^*Lie(\A(-Z))\big)[1]\,.$$

\item Let $L^i_n$ be the $i$th cohomology group of
$$\RR\Gamma\big(C_n,\pi^*_nLie(\A(-Z))\big)\otimes^{\mbb{L}}\QQ_p/\ZZ_p=\RR\Gamma\big(C_n,\pi^*_n Lie(\A(-Z))\big)[1]\,.$$

By the Zariski base change formula (note that the cohomology of the finite locally free module $Lie({\A})(-Z)$ is the same in the \'etale or Zariski site), we have isomorphisms
$$L^i_\infty\simeq L^i_0\otimes W(\kpinf)$$
and, for any $n\geq 0$,
$$L^i_n\simeq L^i_0\otimes W(k_n).$$
In particular, since $L^i_0$ is a finite $\FF_p$-vector space with rank $d(L^i_0)$, we deduce that $L^i_\infty$ is a finite $\kpinf$-vector space while $L^i_n$ is a
finite $k_n$-vector space, both with the same rank $d(L^i_0)$.
\item Let $$N^i_\infty := H^i_{syn}( C_\infty , \pi^*\mathcal{S}_D \otimes \QQ_p/\ZZ_p )$$
be the $i$th cohomology group of $$\RR\Gamma(C_\infty,\pi^*\cS_D)\otimes^{\mbb{L}}\QQ_p/\ZZ_p.$$
\item Let $N^i_n$ be the $i$th cohomology group of $$\RR\Gamma(C_n,\pi_n^*\cS_D)\otimes^{\mbb{L}}\QQ_p/\ZZ_p.$$
\end{enumerate}

The distinguished triangles (\ref{T1}) and (\ref{T2}) induce, by passing to the cohomology, the following long exact sequences:
\begin{gather}
\begin{CD} \label{L1} ... @>>> N^i_\infty @>>> M_{1,\infty}^i @>{\bf 1}-\varphi_{i,\infty}>> M_{2,\infty}^i @>>> ... \end{CD}\\
\begin{CD} \label{L2} ... @>>> L^i_\infty @>>> M_{1,\infty}^i @>{\bf 1}>> M_{2,\infty}^i @>>> ...  \end{CD}
\end{gather}
which are inductive limits of the long exact sequences
\begin{gather}
\begin{CD} \label{L1n} ... @>>> N^i_n @>>>  M_{1,n}^i @>{{\bf 1}-\varphi_{i,n}}>> M_{2,n}^i @>>> ...  \end{CD}\\
\begin{CD} \label{L2n} ... @>>> L^i_n @>>>  M_{1,n}^i @>{\bf 1}>> M_{2,n}^i @>>> ...  \end{CD}
\end{gather}
Note that the cohomology theories $M$ and $N$ are concentrated in degrees 0,1 and 2 and the cohomology theory $L$ is concentrated in degrees 0 and 1.\\

\begin{mytheorem}\label{GIMC1} For $i=0,1,2$ and $j=0,1$, the Pontryagin duals of $N^i_\infty$ and $L^j_\infty$ are finitely generated torsion $\La$-modules.
\end{mytheorem}

\noindent The proof shall be given in Corollaries \ref{M,L} and \ref{1case}. Note that, by \cite{KT03}, $X_p(A/\Kpinf)$ is a submodule of the dual of $N^1_\infty$, so Theorem \ref{GIMC1} shows that under our hypothesis $X_p(A/\Kpinf)$ is also a finitely generated torsion $\La$-module. This was already known by \cite{ot09}, whose argument is simplified in the present paper.

In \cite{KT03}, the following was proved (see \cite[\S3.3.3, \S3.3.4 and \S3.3.5]{KT03}):
\begin{lemma}\label{cofinite}
For $k=1$ or $2$, there exists a map
$$f_k\colon P^i_0[p^{-1}]\lr M^i_{k,0}$$
satisfying the following conditions:
\begin{enumerate}
\item The kernel of $f_k$ is a $\ZZ_p$-lattice in $P^i_0[\frac{1}{p}]$ and the cokernel is a finite group. In particular, $M^i_{1,0}$ and
  $M^i_{2,0}$ are torsion $\ZZ_p$-modules with the same finite corank.
\item The diagrams
    \[\begin{CD}
    P^i_0[\frac{1}{p}] @>{id}>> P^i_0[\frac{1}{p}] && \hspace{60pt} && P^i_0[\frac{1}{p}] @>{p^{-1}F_{i,0}}>> P^i_0[\frac{1}{p}]\\
    @V{f_1}VV @V{f_2}VV  \text{and} && @V{f_1}VV @V{f_2}VV \\
    M^i_{1,0} @>{{\bf 1}}>> M^i_{2,0} && \hspace{60pt} && M^i_{1,0} @>{\varphi_{i,0}}>> M^i_{2,0}
    \end{CD}\]
    commute.
\end{enumerate}
\end{lemma}

\begin{lemma}\label{linalg1} Let $M$ be a torsion $\ZZ_p$-module of cofinite type. Then the Pontryagin dual of $M\otimes_{\ZZ_p} W(\kpinf)$ is a finitely generated $\La$-module. Moreover, we have an isomorphism of $\La$-modules
$$\big(M\otimes_{\ZZ_p} W(\kpinf)\big)^\vee\simeq\La^r\oplus \bigoplus_{i=1}^s \La/(p^{n_i})\,.$$
\end{lemma}

\begin{proof} Let $X_\infty$ denote the Pontryagin dual of $M\otimes_{\ZZ_p} W(\kpinf)$. Then $X_\infty$ is the limit of the projective system of Pontryagin duals $X_n:=\big(M\otimes_{\ZZ_p} W(k_n)\big)^\vee$. Note that the functor $M\rightsquigarrow(M\otimes_{\ZZ_p} W(\kpinf))^\vee$ is exact.

In the case $M=\ZZ/p\ZZ$, we have a projective system $\{X_n=(k_n)^\vee\}_n$ where the transition maps are the Pontryagin dual of the canonical inclusions
$k_n\hookrightarrow k_{n+1}$. Let $\Omega:=\FF[[\Gamma]]\simeq \FF[[T]]$. We have
$$(X_\infty)_\Gamma=X_\infty/TX_\infty=\big((\kpinf)^\Gamma\big)^\vee\simeq\FF\,.$$
Now, by Nakayama's lemma, this implies that $X_\infty$ is a cyclic $\Omega$-module and since $X_\infty$ is infinite we have $X_\infty\simeq\Omega=\La/p\La$.

If $M=\ZZ/p^j\ZZ$, then $X_\infty=(W_j(\kpinf))^\vee={\ilim} (W_j(k_n))^\vee=:Y_j$ (where $W_j$ denotes Witt vectors of length $j$) and we prove by the same argument that $Y_j$ is a cyclic $W_j(\FF)[[\Gamma]]$-module. In particular for each $j$ we have a surjective map:
$$W_j(\FF)[[\Gamma]]\,\lr Y_j\,.$$
We prove by induction on $j$ that $Y_j$ is a free $W_j(\FF)[[\Gamma]]$-module. The case $j=1$ was treated above. Now assume the assertion is true for $j-1$ and consider the commutative diagram of short exact sequences:
$$\begin{CD}
0 @>>> Y_{j-1} @>{\delta}>>  Y_j @>{\epsilon}>> Y_1 @>>> 0 \\
&& @A{\alpha}AA @A{\beta}AA  @A{\gamma}AA \\
0 @>>> W_{j-1}(\FF)[[\Gamma]] @>{\times p}>> W_j(\FF)[[\Gamma]] @>>> \Omega @>>> 0 \,.\end{CD}$$
The upper horizontal line is induced by the exact sequence $\ZZ/p\ZZ\hookrightarrow\ZZ/p^j\ZZ\twoheadrightarrow\ZZ/p^{j-1}\ZZ\,.$ The vertical maps are constructed as follows: define first $\beta$ as the map sending 1 to a generator $x$ of the cyclic $W_j(\FF)[[\Gamma]]$-module $Y_j$. Then $px\in\Ker(\epsilon)=\image(\delta)$. Let $w\in Y_{j-1}$ be such that $px=\delta(w)$ and let $y=\epsilon(x)$. Finally, define $\alpha$ to be the map sending 1 to $w$ and $\gamma$ to be the map sending 1 to $y$. The map $\gamma$ is surjective because $\beta$ and $\epsilon$ are surjective. So $\gamma$ is an isomorphism since $Y_1$ is infinite and the only proper quotients of $\Omega$ are finite. We deduce by the snake lemma that $\alpha$ is surjective and therefore, by the induction hypothesis, an isomorphism. This implies that $\beta$ is also an isomorphism, by the 5-lemma.

The case $M=\QQ_p/\ZZ_p$ is deduced from the case $M=\ZZ/p^j\ZZ$ by passing to the inductive limit in $j$ and the general case follows from the cases $M=\QQ_p/\ZZ_p$ and $M=\ZZ/p^j\ZZ$.
\end{proof}

We deduce from Lemmas \ref{linalg1} and \ref{cofinite}:
\begin{corollary}\label{M,L} \hspace{5pt}
\begin{enumerate}
\item
For any $i\geq 0$, $(M^i_{1,\infty})^\vee$ and $(M^i_{2,\infty})^\vee$ are two $\La$-modules of finite type with the same rank $r_i$ (equal to the $\ZZ_p$-rank of $P^i_0${\em )} and torsion parts isomorphic to $\oplus_{j=1}^s \La/p^{n_j}\La.$
\item For any $i$, $(L^i_\infty)^\vee$ is a finitely generated torsion $\La$-module isomorphic to $\oplus_{j=1}^{d(L^i_0)} \La/p\La\,.$
\item Passing to the Pontryagin dual, the long exact sequences \eqref{L1} and \eqref{L2} induce long exact sequences of $\La$-modules with $\La$-linear operators.
\end{enumerate}
\end{corollary}

\begin{proof} The first assertion is a consequence of Lemmas \ref{linalg1} and \ref{cofinite}. Assertion (2) is proved similarly to the case of $M^i_{k,\infty}$. For the third assertion, note that $\Gamma$ is an abelian group and therefore any $\tau\in\Gamma$ commutes with $\Fr_q$. Hence, the operator ${\bf 1}-\varphi$ is $\La$-linear.
\end{proof}

Next, we are going to prove that $(N^i_\infty)^\vee$ is a finitely generated torsion $\La$-module for $i=0,\dots,2$.

\begin{proposition} \label{0case} The $\La$-modules $(N^0_\infty)^\vee$ and $\left(\Coker({\bf 1}-\varphi_{0,\infty})\right)^\vee$ are $\La$-torsion.
\end{proposition}

\begin{proof} Reasoning as in \cite[\S2.5.2]{KT03} (see \eqref{e:alg-arithm} below), one obtains that the module $(N^0_\infty)^\vee$ is a quotient of $A(\Kpinf)[p^\infty]^\vee$. The latter is a finitely generated $\ZZ_p$-module (and so a torsion $\La$-module), hence the claim for $(N^0_\infty)^\vee$ is proven. As for $(\Coker({\bf 1}-\varphi_{0,\infty}))^\vee$, note that we have an exact sequence of $\La$-modules
$$\begin{CD} 0\lr (\Coker({\bf 1}-\varphi_{0,\infty}))^\vee\ \lr (M_{2,\infty}^0)^\vee @>{({\bf 1}-\varphi_{0,\infty})^\vee}>> (M_{1,\infty}^0)^\vee \lr (N^0_\infty)^\vee \lr 0. \end{CD}$$
Since $(M_{1,\infty}^0)^\vee$ and $(M_{2,\infty}^0)^\vee$ have the same $\La$-rank, it implies that the $\La$-rank of $(\Coker({\bf 1}-\varphi_{0,\infty}))^\vee$ is equal to that of $(N^0_\infty)^\vee$, which is zero.
\end{proof}

We now prove that $(N^1_\infty)^\vee$ is $\La$-torsion.\\

The exact sequence:
\begin{multline}\label{alpha_beta}
\cdots \lra \coh^1_{\mathrm{syn}} (C_\infty,\pi^*\mathcal{S}_{D}) \otimes \QQ_p
\buildrel{\alpha}\over\lra \coh^1_{\mathrm{syn}} (C_\infty,\pi^*\mathcal{S}_{D}
\otimes \QQ_p/\ZZ_p) \\
\buildrel{\beta}\over\lra \coh^2_{\mathrm{syn}} (C_\infty,\pi^*\mathcal{S}_{D})\buildrel{\gamma}\over\lra \coh^2_{\mathrm{syn}} (C_\infty,\pi^*\mathcal{S}_{D})\otimes \QQ_p\lra \cdots
\end{multline}
induces a short exact sequence
\begin{equation}\label{exact-section2}
0\lra \image(\alpha)\lra N^1_\infty\lra \image(\beta)\lra 0.
\end{equation}
By taking the Pontryagin dual of $(\ref{exact-section2})$, we have
\begin{equation}\label{exactbisbis-section2}
0\lra \image(\beta)^\vee \lra (N^1_\infty)^\vee \lra \image(\alpha )^\vee \lra 0,
\end{equation}
where the modules and the morphisms are naturally defined over $\La$.

\begin{lemma}\label{finitevector}
$\coh^1_{\mathrm{syn}} (C_\infty,\pi^*\mathcal{S}_{D})\otimes \QQ_p$ is a finite dimensional
$\QQ_p$-vector space.
\end{lemma}

\begin{proof} The long exact sequence
\begin{multline*}
\cdots
\lra \coh^i_{\mathrm{syn}} (C_\infty,\pi^*\mathcal{S}_{D})\otimes \QQ_p\lra \coh^i_{\mathrm{crys}}\big(C_\infty^\#/W(\kpinf),\pi^*D^0(-Z)\big)\otimes \QQ_p \\
\buildrel{1-\varphi_i}\over\lra \coh^i_{\mathrm{crys}}\big(C_\infty^\#/W(\kpinf),\pi^*D(-Z)\big)\otimes \QQ_p\lra \cdots \end{multline*}
can be rewritten
\begin{multline*}
\cdots \lra \coh^i_{\mathrm{syn}} (C_\infty,\pi^*\mathcal{S}_{D})\otimes \QQ_p\lra \coh^i_{\mathrm{crys}}(C_\infty^\#/W(\kpinf),\pi^*D(-Z))\otimes \QQ_p \\
\buildrel{1-\varphi_i}\over\lra \coh^i_{\mathrm{crys}}(C_\infty^\#/W(\kpinf),\pi^*D(-Z)) \otimes \QQ_p\lra \cdots
\end{multline*}
(because the difference between the middle terms in these sequences is $L^i_\infty\otimes\QQ_p$, which, thanks to Lemma \ref{M,L}, (2) is trivial).\\
We deduce from this long exact sequence the following short exact sequence:
$$ 0\lra \Coker(1-\varphi_0)\lr \coh^1_{\mathrm{syn}} (C_\infty,\pi^*\mathcal{S}_{D})\otimes \QQ_p\lra \Ker(1-\varphi_1)\lra 0. $$
Since $\coh^i_{\mathrm{crys}}(C_\infty^\#/W(\kpinf),\pi^*D(-Z))\otimes\QQ_p$ is a finite dimensional $W(\kpinf)[\frac{1}{p}]$-vector space and \eqref{e:pi0n} holds, the assertion is implied by the following:
\begin{lemma}\label{sblem-section2} Let $V$ be a finite dimensional $\QQ_p$-vector space endowed with a linear operator $\varphi\colon V\rightarrow V$. Then
$1-\varphi\otimes\Fr_q\colon V\otimes_{\ZZ_p}W(\kpinf)\to V\otimes_{\ZZ_p}W(\kpinf)$ is a surjective map whose kernel is a finite dimensional $\QQ_p$-vector space.
\end{lemma}
\begin{proof} One can easily check that the proof of \cite[Lemme 6.2]{EL97} remains true if we replace $\bar{\FF}_p$ by $\kpinf$.
\end{proof}
\end{proof}

\begin{corollary}\label{lemmaN^1}
The group $\image(\alpha)^{\lor}$ of \eqref{exactbisbis-section2} is a free $\ZZ_p$-module of finite rank.
\end{corollary}

\begin{proof}
By (\ref{L1}), $N^1_\infty$ is a torsion $\ZZ_p$-module. Thus, $\image(\alpha)$ is a torsion $\ZZ_p$-module
which is a quotient of $\coh^1_{\mathrm{syn}}(C_\infty,\pi^*\mathcal{S}_{D})\otimes \QQ_p$. We deduce from Lemma \ref{finitevector} that $\image(\alpha)$
is cofree of finite corank $n\leq \dim_{\QQ_p}(\coh^1_{\mathrm{syn}} (C_\infty,\pi^*\mathcal{S}_{D})\otimes \QQ_p)$.
\end{proof}

We now study the term $\image(\beta)$:
\begin{lemma}\label{beta}
The group $\image(\beta )^{\lor}$ of \eqref{exactbisbis-section2} is a finitely generated torsion $\La$-module.
\end{lemma}

\begin{proof}
Note that \eqref{alpha_beta} yields an isomorphism $\image(\beta)\simeq \Ker(\gamma)$.
The kernel of the map
$$\gamma\colon\coh^2_{\mathrm{syn}} (C_\infty,\pi^*\mathcal{S}_{D})\lra \coh^2_{\mathrm{syn}} (C_\infty,\pi^*\mathcal{S}_{D})\otimes \QQ_p$$
is $\coh^2_{\mathrm{syn}} (C_\infty,\pi^*\mathcal{S}_{D})[p^\infty]$.
Recall that we have a short exact sequence:
$$ 0\lra \Coker({\bf 1}-\varphi_{1,\infty})\lra \coh^2_{\mathrm{syn}} (C_\infty,\pi^*\mathcal{S}_{D}) \lra \Ker({\bf 1}-\varphi_{2,\infty})\lra 0. $$
By taking the $p$-power torsion part of this sequence, we have
\begin{equation}\label{exact2-section2}
0\lra \Coker({\bf 1}-\varphi_{1,\infty})[p^\infty] \lra \image(\beta) \lra \Ker({\bf 1}-\varphi_{2,\infty})[p^\infty]\,.
\end{equation}
Taking Pontryagin duals, we have the following:
\begin{equation}\label{seq}
\Ker({\bf 1}-\varphi_{2,\infty})[p^\infty]^\vee \lra \image(\beta)^\vee  \lra \Coker({\bf 1}-\varphi_{1,\infty})[p^\infty]^\vee \lra 0,
\end{equation}
where the modules and the morphisms are defined over $\La$.
By the sequence $(\ref{seq})$, it is enough to show that $\Coker({\bf 1}-\varphi_{1,\infty})[p^\infty]^\vee$ and $\Ker({\bf 1}-\varphi_{2,\infty})[p^\infty]^\vee$ are finitely generated
torsion $\La$-modules. But these two groups are both $p$-torsion subgroups of finitely generated $W(\kpinf)$-modules, so their Pontryagin duals are $\La$-torsion by Lemma \ref{linalg1}.
\end{proof}

\begin{corollary}\label{1case} The groups
$$(N^1_\infty)^\vee \text{, } (\Ker({\bf 1}-\varphi_{1,\infty}))^\vee \text{, } (\Coker({\bf 1}-\varphi_{1,\infty}))^\vee \text{, } (\Ker({\bf 1}-\varphi_{2,\infty}))^\vee \text{ and } (N^2_\infty)^\vee$$
are $\La$-torsion. In particular, $X_p(A/\Kpinf)$ is $\La$-torsion.
\end{corollary}

\begin{proof} The module $(N^1_\infty)^\vee$ is $\La$-torsion by Lemma \ref{beta} and Corollary \ref{lemmaN^1}. By \eqref{L1}, using the short exact sequence
\begin{equation} \label{e:sesNi} 0\lr \Coker({\bf 1}-\varphi_{i-1,\infty})\lr N^i_\infty\lr \Ker({\bf 1}-\varphi_{i,\infty})\lr 0 \end{equation}
we find that $(\Ker({\bf 1}-\varphi_{1,\infty}))^\vee$ is $\La$-torsion. Then by using the exact sequence
$$0\lr \Ker({\bf 1}-\varphi_{1,\infty})\lr M^1_{1,\infty}\lr M^1_{2,\infty}\lr \Coker({\bf 1}-\varphi_{1,\infty})\lr 0$$
we deduce that $(\Coker({\bf 1}-\varphi_{1,\infty}))^\vee$ is $\La$-torsion. Finally, by the exact sequence
$$0\lr \Ker({\bf 1}-\varphi_{2,\infty})\lr M^2_{1,\infty}\lr M^2_{2,\infty}\lr 0$$
we know that $(\Ker({\bf 1}-\varphi_{2,\infty}))^\vee$ is $\La$-torsion. Hence, the short exact sequence
$$0\lr \Coker({\bf 1}-\varphi_{1,\infty})\lr N^2_\infty\lr \Ker({\bf 1}-\varphi_{2,\infty})\lr 0$$
tells us that $(N^2_\infty)^\vee$ is also $\La$-torsion. For the last assertion, note that the surjections $N^1_n\twoheadrightarrow \Sel_{p^\infty}(A/K_n)$ proved in \cite[\S2.5.2]{KT03} induce a surjection $N^1_\infty\twoheadrightarrow \Sel_p(A/\Kpinf)$ by passing to the inductive limit in $n$. But since $(N^1_\infty)^\vee$ is $\La$-torsion, the same assertion holds for $X_p(A/\Kpinf)$.
\end{proof}

\noindent This completes the proof of Theorem \ref{GIMC1}.

\subsection{The characteristic element} \label{su:charelem}
For any finitely generated $\La$-torsion module $M$, we let $f_M\in \La$ be a characteristic element associated with $M$; $f_M$ is defined uniquely up to $\La^\times$.

We set
\begin{equation} \label{e:charelement}
f_{A/\Kpinf}:=\frac{f_{(N^1_\infty)^\vee}f_{(L^0_\infty)^\vee}}{f_{(N^0_\infty)^\vee}f_{(N^2_\infty)^\vee}f_{(L^1_\infty)^\vee}}\in Q(\La)^\times/\La^\times
\end{equation}
and call it the characteristic element associated with $A/K$ relative to the arithmetic $\ZZ_p$-extension of $K$.

\subsubsection{Arithmetic interpretation} \label{ss:bsdinfty} The characteristic element $f_{A/\Kpinf}$ is related to the arithmetic invariants of $A$ as follows. By considering the
$p$-torsion part of the exact sequence in \cite[\S2.5.2]{KT03} and using the fact that the functor ``take the $p$-primary part'' is exact in the category of finite abelian groups,
we deduce an exact sequence:
\begin{equation} \label{e:alg-arithm0} 0\lr N^0_0\lr A(K)[p^\infty]\lr (\oplus_{v\in Z}\M_v)[p^\infty]\lr N^1_0\lr \Sel_{p^\infty}(A/K)\lr 0, \end{equation}
with $\M_v:=A(K_v)/{\A}(\fm_v)$, where $O_v$ and $\fm_v$ are ring of integers and maximal ideal of $K_v$ and
$${\A}(\fm_v):=\Ker\!\big(A(K_v)= {\A}(O_v)\lr {\A}(k(v))\big).$$
Observe that, since we assumed that $A/K$ has semistable reduction, we can take the group $V_v$ defined in \cite[Proposition 5.13]{KT03} to be equal to ${\A}(\fm_v)$. Since $\A/O_v$ is
smooth we have in fact $\M_v=\A(k(v))$. The finite group $\M_v$ is controlled by the short exact sequence:
$$0\,\lr Q_v\lr \M_v\lr \Phi_v\lr 0,$$
where $\Phi_v$ is the (finite) group of components, $\Phi_v:=({\A}/{\A}^0)(k(v))$, and
$$Q_v={\A}^0(O_v)/({\A}(\fm_v)\cap{\A}^0(O_v))={\A}^0(k(v))$$
by Hensel's lemma. Moreover, since $A/K$ has semistable reduction at $v$, we have a short exact sequence:
$$0\,\lr T_v\lr {\A}^0_v\lr B_v\lr 0,$$
where $T_v$ is a torus and $B_v$ an abelian variety over $k(v)$.\\

Since the N\'eron model functor is stable by \'etale base change, we also have for any $n\geq 0$ an exact sequence:
$$ 0\to N^0_n\to A(K_n)[p^\infty]\to(\oplus_{w\in C_n, w|v\in Z}\M_w)[p^\infty]\to N^1_n\to \Sel_{p^\infty}(A/K_n)\to 0, $$
which induces, by passing to the inductive limit in $n$, an exact sequence:
\begin{equation} \label{e:alg-arithm} 0\to N^0_\infty\to A(\Kpinf)[p^\infty]\to (\oplus_{w\in C_\infty, w|v\in Z}\M_w)[p^\infty]\to N^1_\infty\to \Sel_{p^\infty}(A/\Kpinf)\to 0 \end{equation}
and then, by passing to the Pontryagin dual, an exact sequence of finitely generated torsion $\La$-modules. We set
$$\M_n:=(\oplus_{w\in C_n, w|v\in Z}\M_w)[p^\infty]$$
and $\M_\infty:=\varinjlim \M_n$. By multiplicativity of characteristic elements associated with torsion $\La$-modules, we have:
\begin{equation} \label{alg-arith1} \frac{f_{(N^1_\infty)^\vee}}{f_{(N^0_\infty)^\vee}}=\frac{f_{X_p(A/\Kpinf)}f_{\M_\infty^\vee}^{}}{f_{A(\Kpinf)[p^\infty]^\vee}}. \end{equation}

Let $\zeta_1,\dots,\zeta_l$ be the eigenvalues of the Galois actions of $\Fr_q$ on $T_pA(K^{(p)}_\infty)[p^{\infty}]$. Then by \cite[Proposition 2.3.5]{tan10b} we can write
\begin{equation}\label{e:fvee}
f_{A(K^{(p)}_\infty)[p^{\infty}]^\vee}=\prod_{i=1}^l (1-\zeta_i^{-1}\Fr_q^{-1})\,.
\end{equation}

\begin{lemma} \label{l:minfty} Denote by $g(v)$ the dimension of $B_v$ and let $\beta_{1}^{(v)},..., \beta_{2g(v)}^{(v)}\in \bar\QQ_p $ be the eigenvalues of the Frobenius endomorphism $\tF_{B_v}^{deg(v)}$ (Here we denote $\tF_{B_v}$ the absolute Frobenius of $B_v$ and $deg(v):=[k(v):\FF_p]$). Then
\begin{equation}\label{e:fMinfty}
f_{\M_\infty^\vee}=\prod_{v\in Z}\prod_{i=1}^{2g(v)} ({\beta_{i}^{(v)}}-\Fr_v^{-1})\,.
\end{equation}
\end{lemma}

\begin{proof} For $v$ a place of $K$, let $\Gamma_v\subset\Gamma$ denote the decomposition group at $v$ and put $\La_v:=\ZZ_p[[\Gamma_v]]$. Thus we get $\La=\oplus_{\sigma\in \Gamma/\Gamma_v}\sigma\La_v\,.$ For each $v\in Z$, choose a $w_0\in C_\infty$ sitting over $v$. Then
$$\M_{\infty, v}:=\bigoplus_{w\in C_\infty, w|v}\M_w[p^\infty]=\bigoplus_{\sigma\in \Gamma/\Gamma_v} \sigma \M_{w_0}[p^\infty] $$
and hence $\M_{\infty, v}^\vee=\La \otimes_{\La_v} \M_{w_0}[p^\infty]^\vee$. In particular, the characteristic element $f_{\M_{\infty, v}^\vee}$ can be chosen to be that of $\M_{w_0}[p^\infty]^\vee$ over $\La_v$. Also, since $\Phi_v$ is finite and $T_v$ is a torus, $\M_{w_0}[p^\infty]^\vee$ is pseudo-isomorphic to $B_v[p^\infty](k_\infty^{(p)})^\vee$.
For each $v$, we order the eigenvalues $\beta_i^{(v)}$ so that $\beta_i^{(v)}$ is a $p$-adic unit if and only if $i\leq f(v)\leq g(v)$.

Let $\Fr_v\colon x\mapsto x^{q_v}$ denote the Frobenius substitution as an element of $\Gal(\overline{k(v)}/k(v))$ (and also, by abuse of notation, the corresponding element of $\Gal(k_\infty^{(p)}/k(v))$ and of $\Gamma_v$). The product $\beta_{1}^{(v)}\cdot...\cdot\beta_{f(v)}^{(v)}$ is a $p$-adic unit and $\prod_{i>f(v)} ({\beta_{i}^{(v)}}-\Fr_v^{-1})$ is a unit in $\La$.
We claim that $\beta_{1}^{(v)},...,\beta_{f(v)}^{(v)}$ are the eigenvalues of the action of $\Fr_v$ on the Tate module $T_pB_v\,$.
Then by \cite[Proposition 2.3.6]{tan10b} we can write
$$ f_{\M_{\infty, v}^\vee}=\prod_{i=1}^{f(v)} ({\beta_{i}^{(v)}}-\Fr_v^{-1})=\prod_{i=1}^{2g(v)} ({\beta_{i}^{(v)}}-\Fr_v^{-1})$$
and \eqref{e:fMinfty} follows from $f_{\M_\infty^\vee}=\prod_{v\in Z} f_{\M_{\infty,v}^\vee}$

We have to prove the claim. Let $\rho\colon\mathrm{End}_{k(v)}(B_v)\rightarrow \mathrm{End} (T_pB_v)$ denote the $p$-adic representation. Then for every $f\in \mathrm{End}_{k(v)}(B_v)$ the eigenvalues of $\rho(f)$, counting multiplicities, are a portion of those of $f$ (this can be seen e.g.~mimicking the argument in the proof of \cite[Theorem 12.18]{gm13}, with $\Ker(f)$ replaced by its maximal \'etale subgroup). In particular, we can rearrange the order of the $\beta_i^{(v)}$'s so that, for every positive integer $N$, the eigenvalues of $\rho(\tF_{B_v,q_v}^N)=\Fr_v^N$ equal $(\beta_1^{(v)})^N,...,(\beta_{h(v)}^{(v)})^N$ for some $h(v)\leq f(v)$. Let $k(v)_N$ denote the degree $N$ extension of $k(v)$. Letting $\equiv_p$ denote congruence modulo $p$-adic units, we have
$$\prod_{i=1}^{h(v)}(1-(\beta_{i}^{(v)})^N)\equiv_p|B_v[p^\infty](k(v)_N)|\equiv_p\prod_{i=1}^{2g(v)}(1-(\beta_{i}^{(v)})^N)\equiv_p\prod_{i=1}^{f(v)}(1-(\beta_{i}^{(v)})^N)\,,$$
where the second equality is from Weil's formula and the third is from the fact that $1-(\beta_{i}^{(v)})^N$ is a $p$-adic unit if $i>f(v)$. Taking $N$ such that $(\beta_{i}^{(v)})^N\equiv 1$ mod $p$ for all $i<f(v)$, we deduce $h(v)=f(v)$.
\end{proof}

Write $\AA_K$ for the adelic ring. Let $\mu=(\mu_v)_v$ be the Haar measure on $Lie({\A})( \AA_K)$ such that
$$\mu_v(Lie({\cal A})(O_v)):=1$$ for every $v$ and let $\alpha_v$ denote the Haar measure on $A(K_v)=\A(O_v)$ such that for $n\geq 1$
$$\alpha_v(\A(\fm_v^n)):=\mu_v (Lie({\cal A})(\fm_v^n)).$$
Then since $|\A(O_v)/\A(\fm_v)|=|\M_v|$ and $Lie({\cal A})(O_v)/Lie({\cal A})(\fm_v)\simeq (O_v/\fm_v)^g$, we have
\begin{equation}\label{e:alpgaM}
\alpha_v(\A(O_v))=|\M_v|\cdot q_v^{-g}.
\end{equation}
By \cite[p.552]{KT03} we have the relation:
\begin{equation} \label{e:tamagawa} |\M_0|\cdot |L^0_0|\cdot |L^1_0|^{-1}= \mu(Lie(\A)( \AA_K)/Lie(\A)( K))^{-1}\cdot\prod_{v\in Z}\alpha_v(\A(O_{v}))\,. \end{equation}
Thus, by \eqref{e:alpgaM} and \eqref{e:tamagawa}
\begin{equation}\label{e:tamagawa2} |L^0_0|\cdot|L^1_0|^{-1}=q^{-g\deg (Z)}\cdot \mu(Lie({\cal A})( \mathbb{A}_K)/ Lie({\cal A})( K))^{-1}. \end{equation}

Next, we choose a a basis $e_1,...,e_g$ of the $K$-vector space $Lie(A)(K)=Lie({\A})(K)$. Then for every $v$ the exterior product $e:=e_1\wedge\cdots\wedge e_g$
determines the Haar measure $\mu_v^{(e)}$ on $ Lie({\A})(K_v)$ that has measure $1$ on the compact subset $\fL(O_v):=O_ve_1+\cdots + O_ve_g$. Similarly, if we choose a a basis $f_1,...,f_g$ of $Lie({\A})(O_v)$ over $O_v$, then the exterior product $f_v:=f_{1 v}\wedge\cdots\wedge f_{g v}$ actually determines the Haar measure $\mu_v$. Define the number $\delta$ by
\begin{equation}\label{e:deltaratio}
q^{-\delta}:=\prod_{\text{all}\; v} \frac{\mu_v^{(e)}(Lie({\A})(O_v))}{\mu_v(Lie({\A})(O_v))}= \prod_{\text{all}\; v} \mu_v^{(e)}(Lie({\A})(O_v))\,, \end{equation}
so that the Haar measures $\mu$ and $\mu^{(e)}$ are related by $\mu^{(e)}=q^{-\delta}{\mu}\,.$
\footnote{If $g=1$ and $\Delta$ denote the global discriminant, then $\delta=\frac{\deg (\Delta)}{12}$ (see e.g.~ \cite[eq.~(9)]{tan95v}).}
By a well-known computation (see e.g.~\cite[VI, Corollary 1 of Theorem 1]{We73}) one finds
$$\mu^{(e)}(Lie({\cal A})( \mathbb{A}_K)/Lie({\cal A})(K))=q^{g(\kappa-1)},$$
with $\kappa$ the genus of $C/\FF$, whence, by \eqref{e:tamagawa2} and \eqref{e:deltaratio},
\begin{equation}\label{e:tamagawa3}
|L^0_0|\cdot|L^1_0|^{-1}=q^{-g(\deg (Z)+\kappa-1)-\delta}.
\end{equation}

\begin{lemma}\label{l:l0l1}
Under the above notation we can write
$$\frac{f_{(L_\infty^0)^\vee}}{f_{(L_\infty^1)^\vee}}=q^{-g(\deg (Z)+\kappa-1)-\delta}.$$
\end{lemma}

\begin{proof}
Since $L_\infty^i\simeq L_0^i\otimes_{\ZZ_p} W(k_\infty^{(p)})$, the lemma follows from \eqref{e:tamagawa3} and Lemma \ref{linalg1} (as well as its proof).
\end{proof}

Finally, by \cite[2.5.3]{KT03}, we have for any $n\geq 0$ an isomorphism:
\begin{equation} \label{e:KT253} N^2_n\simeq \Sel_{\ZZ_p}(A^t/K_n)^\vee, \end{equation}
where $\Sel_{\ZZ_p}(.):=\varprojlim \Sel_{p^n}(.)$ denotes the compact Selmer group as in \cite[\S2.3]{KT03}. These isomorphisms induce, when passing to the inductive limit, an isomorphism
\begin{equation} \label{e:N2compactSel} (N^2_\infty)^\vee\simeq \Sel_{\ZZ_p}(A^t/\Kpinf):=\ilim \Sel_{\ZZ_p}(A^t/K_n)\,. \end{equation}
Let $T_p(A(K^{(p)}_\infty)):=\varprojlim A(K^{(p)}_\infty)[p^{m}]$ denote the Tate-module of $A(K^{(p)}_\infty)$.

\begin{proposition} \label{p:N2trivial}
If $A_{p^\infty}(\Kpinf)$ is a finite group, then $N^2_\infty=0$.
In general, we can write
$$f_{(N^2_{\infty})^\vee}=f_{T_p(A^t(K^{(p)}_\infty))}.$$
\end{proposition}

\begin{proof} For simplicity denote $D_n:=A^t(K_n)[p^{\infty}]$. Also, write $V_n:=\varprojlim_m \Sel_{p^{\infty}}(A^t/K_n)[p^m]$. Since $\Sel_{p^{\infty}}(A^t/K_n)$ is cofinitely generated over $\ZZ_p\,$, actually
$$V_n=\varprojlim_m \Sel_{p^{\infty}}(A^t/K_n)_{div}[p^m].$$

For each $m$ we have the exact sequence
$$\xymatrix{0 \ar[r] & D_n/p^m D_n \ar[r] & \Sel_{p^{m}}(A^t/K_n) \ar[r] & \Sel_{p^{\infty}}(A^t/K_n)[p^m]\ar[r] & 0},$$
which induces, by taking $m\rightarrow \infty$, the exact sequence
$$\xymatrix{0 \ar[r] & D_n \ar[r] & \Sel_{\ZZ_p}(A^t/K_n) \ar[r] & V_n\ar[r] & 0}.$$
For each $n$, let $\Sel_{p^{\infty}}(A/K_n)_{div}$ denote the $p$-divisible part of $\Sel_{p^{\infty}}(A/K_n)$ and let $Y_p(A/K_n)$ be its Pontryagin dual. Also, let $Y_p(A/\Kpinf)$ denote the projective limit of $\{Y_p(A/K_n)\}_n$, so that
$$Y_p(A/\Kpinf) = \Sel_{div}(A/\Kpinf)^\vee,$$
where
$$\Sel_{div}(A/\Kpinf):=\varinjlim_{n}\Sel_{p^{\infty}}(A/K_n)_{div}\,.$$

Since $\Sel_{p^\infty}(A^t/\Kpinf)$ is cotorsion over $\La$, the divisible subgroup $\Sel_{div}(A^t/\Kpinf)$ must be cofinitely generated over $\ZZ_p\,$. Thus, for $n$ sufficiently large the restriction map $\Sel_{p^{\infty}}(A^t/K_n)_{div}\rightarrow \Sel_{div}(A^t/\Kpinf)$ is surjective. The kernel of this map is finite (note that we can show as in \cite{gr03} that the groups $\coh^1\!\big(\Gamma^{(n)},A^t_{p^{\infty}}(\Kpinf)\big)$ are finite). It follows that if $n$ is sufficiently large then the restriction map $\Sel_{p^{\infty}}(A^t/K_n)_{div}\rightarrow \Sel_{p^{\infty}}(A^t/K_r)_{div} $ is an isomorphism for every $r>n$, and hence $V_r$ can be identified with $V_n$. This implies $\varprojlim_n V_n=0$ as the map $V_r\rightarrow V_n$ becomes multiplication by $p^{r-n}$ on $V_r$ for sufficiently large $r$, $n$. Hence, \eqref{e:N2compactSel} and the above exact sequence yield
$$(N^2_\infty)^\vee=\varprojlim_n \Sel_{\ZZ_p}(A^t/K_n)=\varprojlim_n A^t(K_n)[p^{\infty}].$$
This proves the first assertion, while the second follows from \cite[Proposition 2.3.5]{tan10b}.
\end{proof}


Since $A$ and $A^t$ are isogenous, the Galois actions of $\Fr_q$ on both $T_p(A(K^{(p)}_\infty))$ and $T_p(A^t(K^{(p)}_\infty))$ have the same eigenvalues.
Then, by \cite[Proposition 2.3.5]{tan10b} again, we can write
\begin{equation}\label{e:ftp}
f_{T_p(A^t(K^{(p)}_\infty))}=\prod_{i=1}^l (1-\zeta_i^{-1}\Fr_q)\,.
\end{equation}

\subsubsection{The value of $\star_{A,\Kpinf}$} Summarizing the above and observing that both \eqref{e:fMinfty} and \eqref{e:ftp} are units, we can make more precise the statement of Theorem \ref{t:summary}. Recall that, in the notation of our Introduction, $c_{A/\Kpinf}=f_{X_p(A/K_{\infty}^{(p)})}\,$.

\begin{proposition}\label{p:summarize} Let notation be as above. Then we can write
$$f_{A/K_{\infty}^{(p)}}\equiv_{\La^\times}\star_{A,K_{\infty}^{(p)}}\cdot c_{A/\Kpinf}.$$
with
$$\star_{A,K_{\infty}^{(p)}}=\frac{q^{-g(\deg(Z)+\kappa-1)-\delta}\cdot\prod_{v\in Z}\prod_{i=1}^{2g(v)} (\beta_i^{(v)}- \Fr_v^{-1})}{\prod_{i=1}^l(1-\zeta_i^{-1}\Fr_q)(1-\zeta_i^{-1}\Fr_q^{-1})}.$$
\end{proposition}

\section{The Main Conjecture} \label{s:semistab}

In this section we give a geometric analogue of the Iwasawa Main Conjecture of abelian varieties with semistable reduction. We keep the notations and the hypotheses of Section \ref{s:genord}.

\subsection{Interpolation and the Main Conjecture}

\subsubsection{The modules $P^i_\infty$} \label{ss:Pinfty}
Let $\QQ_{p,n}$ denote the fraction field of $W(k_n)$ and $\QQ_{p,\infty}:=\cup_n \QQ_{p,n}$.

\begin{lemma}\label{induce} For any $i$, the map
$$\big({\bf 1}\otimes W(\kpinf)\big)\!^\vee\left[\frac{1}{p}\right]\colon(M^i_{2,\infty})^\vee\left[\frac{1}{p}\right]\lr (M^i_{1,\infty})^\vee\left[\frac{1}{p}\right]$$
is an isomorphism induced by the identity on $(P^i_0[\frac{1}{p}]\otimes_{\QQ_p} \QQ_{p,\infty})^\vee$.
\end{lemma}

\begin{proof} The first assertion follows from the exact sequence \eqref{L2} and from Corollary \ref{M,L} (2). To show that the map is induced by the identity on $(P^i_0[\frac{1}{p}]\otimes_{\QQ_p}
\QQ_{p,\infty})^\vee$, note that, by Lemma \ref{cofinite} (1), for all $n$ and for $k=1,2$ there exist some exact sequences:
$$0\to B^i_0\otimes_{\ZZ_p} W(k_n)\to P^i_0[p^{-1}]\otimes_{\QQ_p} \QQ_{p,n}\to M^i_{k,0}\otimes_{\ZZ_p} W(k_n)\to C^i_0\otimes_{\ZZ_p} W(k_n) \to 0$$
where $B^i_0$ is a lattice of $P^i_0[\frac{1}{p}]$ and $C^i_0$ a finite abelian $p$-group.

Passing to the inductive limit in $n$, then to the Pontryagin dual, and finally by inverting $p$, we obtain an exact sequence:
\begin{equation}\label{exact1} 0\lr (M^i_{k,\infty})^\vee[p^{-1}]\lr (P^i_0[p^{-1}]\otimes_{\QQ_p}\QQ_{p,\infty})^\vee \lr (\dlim B^i_n)^\vee[p^{-1}]\lr 0,\end{equation}
since $\big(\varinjlim C^i_0\otimes_{\ZZ_p} W(k_n)\big)^\vee$ is $p$-torsion (actually, by Lemma \ref{linalg1}, it is a finite direct sum of $(\ZZ/p^{n_j}\ZZ)[[\Gamma]]$).
In particular, $(M^i_{k,\infty})^\vee[\frac{1}{p}]$ is a submodule of $(P^i_0[\frac{1}{p}]\otimes_{\QQ_p} \QQ_{p,\infty})^\vee$ for $k=1,2$ and the map
$({\bf 1}\otimes W(\kpinf))^\vee[\frac{1}{p}]$ is compatible with the identity on $(P^i_0[\frac{1}{p}]\otimes_{\QQ_p} \QQ_{p,\infty})^\vee$ by Lemma \ref{cofinite}, (2).
\end{proof}

As a consequence of the previous lemma, we denote $P^i_\infty$ the module $(M^i_{2,\infty})^\vee[\frac{1}{p}]$, endowed with the operator $\Phi_i:=\big({\bf 1}\otimes W(\kpinf))^\vee[\frac{1}{p}]\big)^{-1}\circ\big((\varphi_{i,0}\otimes\Fr_q)^\vee[\frac{1}{p}]\big)$. First, we define the $p$-adic $L$-function.

\subsubsection{The $p$-adic $L$-function} \label{su:padicL} By Corollary \ref{M,L}, (1), we know that the $P^i_\infty$'s are free $\La[\frac{1}{p}]$-modules of finite rank. We set
$$\cL_{A/\Kpinf}:=\prod_{i=0}^2 {\det}_{\La[\frac{1}{p}]}\left(id-\Phi_i,P^i_\infty\right)^{(-1)^{i+1}}$$
and call it the $p$-adic $L$-function associated with $A/K$ relative to the arithmetic $\ZZ_p$-extension of $K$.



\begin{lemma}\label{interpol} Let $E$ be a finite extension of $\QQ_p$ and $\omega\colon\Gamma\to E^\times$ an Artin character factoring through $\Gamma_N$. Then we have
$$\omega(\cL_{A/\Kpinf})=\prod_{i=0}^2 {\det}_{E}\left(id-p^{-1}F_{i,0}\otimes m_{\omega(\Fr_q)},P^i_0[\frac{1}{p}]\otimes_{\QQ_p}E\right)^{(-1)^{i+1}}$$
where $m_{\omega(\Fr_q)}\colon E\to E$ is the multiplication by $\omega(\Fr_q)$.
\end{lemma}

\begin{proof} We have for any $i$,
$$\omega({\det}_{\La[\frac{1}{p}]}(id-\Phi_i,P^i_\infty))={\det}_E(id-\Phi_i\otimes id_E,P^i_\infty\otimes_{\La[\frac{1}{p}]}E).$$
(where, as usual, $E$ is a $\La[\frac{1}{p}]$-module via $\omega$).

Since $\omega$ factors through $\Gamma_N$, we have
\begin{eqnarray*}P^i_\infty\otimes_{\La[\frac{1}{p}]}E&=&
P^i_\infty\otimes_{\La[\frac{1}{p}]}\QQ_p[\Gamma_N]\otimes_{\QQ_p[\Gamma_N]}E\\
                                   &=&(P^i_\infty)_{\Gamma^{(N)}}\otimes_{\QQ_p[\Gamma_N]}E\\
                                   &=&\left((M^i_{k,0}\otimes W(\kpinf))^{\Gamma^{(N)}}\right)\!^\vee[p^{-1}]\otimes_{\QQ_p[\Gamma_N]}E\\
                                   &=&(M^i_{k,0}\otimes W(k_N))^\vee[p^{-1}]\otimes_{\QQ_p[\Gamma_N]}E\\
                                   &=&(P^i_0[p^{-1}]\otimes_{\QQ_p}\QQ_{p,N})^\vee\otimes_{\QQ_p[\Gamma_N]}E,
\end{eqnarray*}
where the last equality is a consequence of Lemma \ref{cofinite}. An operator and its dual share the same determinant: hence we are reduced to compute ${\det}_E((id-p^{-1}F_{i,0}\otimes\Fr_q) \otimes id_E)$ on $P^i_0[\frac{1}{p}]\otimes_{\QQ_p}\QQ_{p,N}\otimes_{\QQ_p[\Gamma_N]}E\,.$ By the normal basis theorem $\QQ_{p,N}\otimes_{\QQ_p[\Gamma_N]}E$ endowed with its $E$-endomorphism $\Fr_q\otimes id_E$
is isomorphic to $E$ endowed with the endomorphism $m_{\omega(\Fr_q)}$ and so the assertion is clear.
\end{proof}

\subsubsection{Twisted Hasse-Weil $L$-function} \label{ss:twistedHW} Let $\omega\colon\Gamma\to E_0^\times$ be any character factoring through $\Gamma_N$, with $E_0$ some totally ramified finite extension of $\QQ_{p,0}$ endowed with a Frobenius operator $\sigma$ which acts trivially on $\QQ_{p,0}$ and on $\omega(\Gamma)$. Then we can see $\omega$ as a one-dimensional $E_0$-representation of the fundamental group of $U$, having finite local monodromy. By \cite[Theorem 7.2.3]{Ts98} this representation corresponds to a unique constant unit-root overconvergent isocrystal $U(\omega)^\dagger$ over $\FF/E_0$ endowed with $m_\omega(\Fr_q)$ as Frobenius operator. Let
$$pr_1^*\colon F^a\hbox{-}iso^\dagger(U_{/\FF}/\QQ_{p,0})\lr F^a\hbox{-}iso^\dagger(U_{/\FF}/E_0)$$
and
$$pr_2^*\colon F^a\hbox{-}iso^\dagger(\FF/E_0)\lr F^a\hbox{-}iso^\dagger(U_{/\FF}/E_0),$$
denote the two restriction functors in the categories of overconvergent isocrystals endowed with Frobenius. From $I^\dagger\in F^a\hbox{-}iso^\dagger(U_{/\FF}/\QQ_{p,0})$ we obtain the $F^a$-isocrystal $pr_1^*I^\dagger\otimes pr_2^*U(\omega)^\dagger$ endowed with the natural Frobenius induced by the Frobenius on $I^\dagger$ and the one on $U(\omega)^\dagger$.

\begin{definition} Let $I^\dagger\in F^a\hbox{-}iso^\dagger(U_{/\FF}/\QQ_{p,0})$. Then we set $$L(U,I^\dagger,\omega,t):=L(U,pr_1^*I^\dagger\otimes pr_2^*U(\omega)^\dagger,t)$$ where the right hand term is the classical $L$-function associated with the $F^a$-isocrystal $pr_1^*I^\dagger\otimes pr_2^*U(\omega)^\dagger$, as defined in \cite{EL93}. We call the function $L(U,I^\dagger,\omega,t)$ the $\omega$-twisted $L$-function of $I^\dagger$.
\end{definition}

Recall (\cite[Th\'eor\`eme 6.3]{EL93}) that this is a rational function in the variable $t$ and we have
$$L(U,I^\dagger,\omega,t)=\prod_{i=0}^2\det(1-t\varphi_i,\coh^i_{rig,c}(U/E_0,pr_1^*I^\dagger\otimes pr_2^*U(\omega)^\dagger))^{(-1)^{i+1}}.$$

\subsubsection{Interpolation} Recall (\cite[IV]{KT03}) that $D$, the log Dieudonn\'e crystal associated with our semistable abelian variety $A/K$, induces an overconvergent $F^a$-isocrystal $D^\dagger$ over $U/\QQ_{p,0}$ and that we have a canonical isomorphism:
\begin{equation} \label{e:PHrig} P^i_0[p^{-1}]\simeq \coh^i_{rig,c}(U/\QQ_{p,0},D^\dagger) \end{equation}
compatible with the Frobenius operators.

Moreover, we have, by \cite[3.2.2]{KT03},
$$L(U,D^\dagger,q^{-s})=L_Z(A,s),$$
where $L_Z(A,s)$ is the Hasse-Weil $L$-function of $A$ without Euler factors outside $U$. In fact, more generally, one can show that for any character $\omega\colon\Gamma\to \CC^\times$ we have
\begin{equation} \label{e:LD+=LA} L(U,D^\dagger,\omega,q^{-s})=L_Z(A,\omega,s), \end{equation}
since the Euler factors on both sides can be written as $\prod_{v\in U} (1-\omega([v])\varepsilon_{i,v}q^{-s})^{-1}$, where the $\varepsilon_{i,v}$'s are the eigenvalues of the arithmetic Frobenius at $v$ acting on $T_\ell(A)$ (or, equivalently, of the geometric Frobenius acting on the fibre at $v$ of $D^\dagger$). These eigenvalues don't depend on $\ell$, as results of \cite{KM} (for the details see e.g. the proof of \cite[Corollaire 1.4]{Tr}, where this independence is used to deduce the equality of different definitions of $L$-functions).

The K\"unneth formula for rigid  cohomology (formula (1.2.4.1) in \cite{Ke06}) implies:

\begin{lemma} Let $\omega\colon\Gamma\to E_0^\times$ be as in \S\ref{ss:twistedHW}. There is an isomorphism of $E_0$-vector spaces compatible with Frobenius operators:
$$\coh^i_{rig,c}(U/E_0,pr_1^*D^\dagger\otimes pr_2^*U(\omega)^\dagger)\simeq \coh^i_{rig,c}(U/\QQ_{p,0},D^\dagger)\otimes E_0.$$
\end{lemma}

\noindent Together with Lemma \ref{interpol}, \eqref{e:PHrig} and \eqref{e:LD+=LA}, this immediately yields the following:

\begin{mytheorem}\label{GIMC2} For any Artin character $\omega\colon\Gamma\to \bar\QQ_p^\times$, we have
\begin{equation} \label{e:GIMC2} \omega(\cL_{A/\Kpinf})=L_Z(A,\omega,1). \end{equation}
\end{mytheorem}

\begin{remark}{\em In order to discuss $\omega(\cL_{A/\Kpinf})$, one needs to know that the denominator of $\cL_{A/\Kpinf}$ is not killed by $\omega$. Actually, we are going to see (formula \eqref{e:detfr} below) that $\cL_{A/\Kpinf}$ is an alternating product of terms $1-\alpha_{ij}\Fr_q$. By \cite[Theorem 5.4.1]{Ke06b}, the coefficients $\alpha_{ij}$ are Weil numbers of weight respectively $-1$ (for $i=0$) and 1 (for $i=2$): in particular their complex absolute values do not include 1. Hence the left-hand side of \eqref{e:GIMC2} is well defined.}\end{remark}

\subsubsection{The Main Conjecture} Finally, we prove the Iwasawa Main Conjecture in this setting.
\begin{mytheorem} \label{t:GIMC3} Let $A$ be an abelian variety with at worst semistable reduction relative to the arithmetic extension $\Kpinf/K$.
We have the following equality in $Q(\La)^\times/\La^\times${\em :}
$$\cL_{A/\Kpinf}=f_{A/\Kpinf}\,.$$
\end{mytheorem}

\begin{proof}
For any morphism $g\colon M\to N$ of $\La$-modules whose kernel and cokernel are both torsion $\La$-modules, we denote by $char(g)$ the element $$f_{\Coker(g)}\cdot f_{\Ker(g)}^{-1}\in Q(\La)^\times/\La^\times.$$
Dualizing the exact sequence \eqref{e:sesNi} we get
$$0\lr (\Ker({\bf 1}-\varphi_{i,\infty}))^\vee\lr (N^i_\infty)^\vee \lr (\Coker({\bf 1}-\varphi_{i-1,\infty}))^\vee \lr 0\,$$
which, remembering that $\Ker(h^\vee)=(\Coker(h))^\vee$, implies
$$f_{(N^i_\infty)^\vee}=f_{\Coker(({\bf 1}-\varphi_{i,\infty})^\vee)}f_{\Ker(({\bf 1}-\varphi_{i-1,\infty})^\vee)}\,.$$
Similarly, \eqref{L2} yields $f_{(L^i_\infty)^\vee}=f_{\Coker({\bf 1}_i^\vee)}f_{\Ker({\bf 1}_{i-1}^\vee)}$. Replacing in \eqref{e:charelement}, we obtain
$$f_{A/\Kpinf}=\frac{f_{\Coker(({\bf 1}-\varphi_{1,\infty})^\vee)}f_{\Ker(({\bf 1}-\varphi_{0,\infty})^\vee)}}{f_{\Coker(({\bf 1}-\varphi_{0,\infty})^\vee)}f_{\Coker(({\bf 1}-\varphi_{2,\infty})^\vee)}f_{\Ker(({\bf 1}-\varphi_{1,\infty})^\vee)}}\cdot\frac{f_{\Coker({\bf 1}_0^\vee)}}{f_{\Coker({\bf 1}_1^\vee)}f_{\Ker({\bf 1}_{0}^\vee)}}\,.$$
Since $\Ker(({\bf 1}-\varphi_{i,\infty})^\vee)$ and $\Coker(({\bf 1}-\varphi_{i,\infty})^\vee)$ are $\La$-torsion modules by Proposition \ref{0case} and Corollary \ref{1case}, this can be rewritten as
$$f_{A/\Kpinf}=\frac{char(({\bf 1}-\varphi_{1,\infty})^\vee)\cdot char({\bf 1}_1^\vee)^{-1}}{char(({\bf 1}-\varphi_{0,\infty})^\vee)\cdot char({\bf 1}_0^\vee)^{-1}\cdot char(({\bf 1}-\varphi_{2,\infty})^\vee)\cdot char({\bf 1}_2^\vee)^{-1}}$$
On the other hand, $\cL_{A/\Kpinf}$ is defined as an alternating product of determinants of
$$id-\Phi_i=({\bf 1}_i^\vee)^{-1}\circ({\bf 1}^\vee-\varphi_{i,\infty})\,.$$
Thus Theorem \ref{t:GIMC3} becomes an immediate consequence of the following lemma (whose proof is an easy exercise which we omit):
\begin{lemma} Let $g,h\colon M\rightarrow N$ be two homomorphisms of finitely generated $\La$-modules with torsion kernel and cokernel: then
$${\det}_{Q(\La)}(g_{Q(\La)}h_{Q(\La)}^{-1})=char(g)char(h)^{-1}\,.$$
\end{lemma}
\end{proof}

\subsection{Euler characteristic} \label{su:Eulerchar}
After identifying $\La$ with $\ZZ_p[[T]]$, the characteristic element of $M$ can be written $f_M(T)=T^rf(T)$, with $f(T)\in \ZZ_p[[T]]$ such that $T$ does not divide $f(T)$. We call $r$ the order of $f_M$ and $f(0)$ the leading term of $f_M$ (note that $f(0)$ is defined up to $\ZZ_p^\times$, since the choice of a different isomorphism $\La\simeq\ZZ_p[[T]]$ changes $T$ by a unit in $\La$).
In the following, we are going to compute order and leading term of $\cL_{A/\Kpinf}$ and compare them with those of the classical $L$-function.

\subsubsection{Generalized Euler characteristic} We recall the definition of the generalized $\Gamma$-Euler characteristic. Let $M$ be a finitely generated torsion $\La$-module and let $g_M\colon M^\Gamma\to M_\Gamma$ denote the composed map $M^\Gamma\hookrightarrow M\to M_\Gamma$, where the first map is the canonical inclusion and the second map the canonical projection. Then we say that $M$ has finite generalized $\Gamma$-Euler characteristic, denoted $char(\Gamma,M)$, if $\Ker(g_M)$ and $\Coker(g_M)$ are finite groups and in this case we set
$$char(\Gamma,M):=\frac{|\Coker(g_M)|}{|\Ker(g_M)|}.$$
By the identifications $(M^\vee)^\Gamma=(M_\Gamma)^\vee$ and  $(M^\vee)_\Gamma=(M^\Gamma)^\vee$, we see that $g_{M^\vee}$ is the dual of $g_M$ and hence
\begin{equation} \label{e:EulerMdual} char(\Gamma,M^\vee)=char(\Gamma,M)^{-1} \end{equation}
if one of them is defined.

\subsubsection{Twisted Euler characteristic} \label{ss:twistEuler}
Let $\omega\colon\Gamma\to \cO^\times$ be an Artin character, with $\cO$ the ring of integers of some finite extension of $\QQ_p$, and $M$ be a finitely generated torsion $\La$-module. Let $\omega^*\colon \La_\cO\rightarrow \La_\cO$ be the automorphism $\gamma\mapsto \omega(\gamma)^{-1}\gamma$
and denote $M_\cO(\omega):=\La_\cO\otimes_{\La_\cO}M_\cO$, where we see $\La_\cO$ as $\La_\cO$-module via $\omega^*$ .
Then $M_\cO(\omega)$ has again a structure of finitely generated torsion $\La$-module.

Assuming that $M_\cO(\omega)$ has finite generalized $\Gamma$-Euler characteristic, we denote $\pounds_\omega(f_M)$ the leading term of $f_{M_\cO(\omega)}$ and $ord_\omega(f_M)$ the order of $f_{M_\cO(\omega)}$. We have the following result (compare \cite[Lemma 2.11]{zer09} and also \cite[Prop.~3.19]{BV06}):

\begin{lemma}\label{recall} Let $M$ be a finitely generated torsion $\La$-module with characteristic element $f_M\in \La/\La^\times$ and let $\omega\colon\Gamma\to \cO^\times$ be a character. Let $d_\cO:=\left[\cO:\ZZ_p\right]$. Then
$$\rank_{\ZZ_p}\left(M_\cO(\omega)^\Gamma\right)=\rank_{\ZZ_p}\left(M_\cO(\omega)_\Gamma\right)\leq ord_\omega\left(f_M\right)=ord\left(\omega^*(f_M)\right),$$
with equality if and only if $M_\cO(\omega)$ has finite generalized $\Gamma$-characteristic and in this case we have
$$char\left(\Gamma,M_\cO(\omega)\right)=|\pounds_\omega(f_M)|_p^{-d_\cO}=|\omega^*(f_M)(0)|_p^{-d_\cO}.$$
\end{lemma}

\begin{proof} First, note that if $M\sim\La_\cO/f\La_\cO$ then $M_\cO(\omega)\sim\La_\cO/\omega^*(f)\La_\cO$. It is an easy exercise to check that if $M$ and $N$ are pseudo-isomorphic $\La_\cO$-modules, then they have the same Euler characteristic (for a hint, see \cite[Lemma 3.5]{css03}). Besides, the Euler characteristic is multiplicative: hence we are reduced to compute it for the case $M=\La_\cO/f\La_\cO$, with $f$ a power of some prime $\xi\in\La_\cO\simeq\cO[[T]]$, and in the rest of the proof we will assume we are in this situation. Then if $f=T$ the map $g_M$ is the identity, while if $f=T^i$ for some $i>1$ we have $g_M=0$ and $M^\Gamma\simeq\cO\simeq M_\Gamma$. Finally, if $f$ is coprime with $T$ we get $M^\Gamma=0$ and
$$M_\Gamma\simeq\La_\cO/(f,T)\simeq\cO/f(0)\cO.$$
Now just remember that, by basic number theory, $|\cO/x\cO|=|x|_p^{-d_\cO}$ for any $x\in\cO$.
\end{proof}

\begin{lemma}\label{L(chi)} Let $\omega\colon\Gamma\to \cO^\times$ and $d_\cO$ be as in the previous lemma. Then, for $j=0,1$, we have
$$char\left(\Gamma,(L^j_\infty)^\vee(\omega)\right)=p^{d_\cO d(L^j_0)}$$
and
$$\rank_{\ZZ_p}\left(\left((L^j_\infty)^\vee(\omega)\right)^\Gamma\right)=0.$$
\end{lemma}

\begin{proof} By Corollary \ref{M,L}, (2), for $j=0,1$, $(L^j_\infty)^\vee$ is a finite direct sum of copies of $\La/p\La$ and therefore $(L^j_\infty)^\vee(\omega)$ is a finite direct sum of copies of $\La_{\cO}/p\La_{\cO}$ and the assertion is clear.
\end{proof}

We deduce now from Lemma \ref{recall}, Lemma \ref{L(chi)}, Theorem \ref{t:GIMC3} and \eqref{e:tamagawa3} the following result:

\begin{mytheorem} \label{chi-Iwasawa-BSD} Let $\omega\colon\Gamma\to \cO^\times$ be a character.
Assume that, for $i=0,1,2$, the $\La$-modules $(N^i_\infty)^\vee(\omega)$ have finite generalized $\Gamma$-Euler characteristic. Then
$$ord_\omega\big(\cL_{A/\Kpinf}\big)=\sum_{i=0}^2 (-1)^{i+1}\rank_{\ZZ_p}\left((N^i_\infty)^\vee(\omega)\right)^\Gamma$$
and
$$|\pounds_\omega(\cL_{A/\Kpinf})|_p^{-d_\cO}= q^{-d_\cO(\delta+g(\deg(Z)+\kappa-1))} \prod_{i=0}^2 char\left(\Gamma,(N^i_\infty)^\vee(\omega)\right)^{(-1)^{i+1}}.$$
\end{mytheorem}

If $\omega$ is the trivial character, we can obtain more precise results. We consider first the problem of finiteness of the generalized $\Gamma$-Euler characteristic.

\subsubsection{Hochschild-Serre spectral sequence} \label{specseq}  Since $\Gamma$ has cohomological dimension one, the natural Hochschild-Serre spectral sequence
\begin{equation} \label{e:HochSerre} \coh^i(\Gamma,N^j_\infty)\Rightarrow N_0^{i+j} \end{equation}
induces (\cite[Appendix B]{mil80}) the following two exact sequences,
\begin{equation}\label{ses0}
0\lr (N^0_\infty)_\Gamma\lr N^1_0 \buildrel{\beta}\over\lr (N^1_\infty)^\Gamma\lr 0
\end{equation}
and
\begin{equation}\label{ses1}
0\lr (N^1_\infty)_\Gamma \buildrel{\gamma}\over\lr N^2_0\lr (N^2_\infty)^\Gamma\lr 0.
\end{equation}

\begin{lemma}\label{N02} The groups $(N^i_\infty)_\Gamma$ and $(N^i_\infty)^\Gamma$ are finite for $i=0$ and $i=2$. In particular, we have
\begin{equation} \label{e:ranks} \rank_{\ZZ_p}(N^1_0)^\vee = \rank_{\ZZ_p}\big((N^1_\infty)^\vee\big)_\Gamma = \rank_{\ZZ_p}(X_p(A/K)). \end{equation}
\end{lemma}

\begin{proof} Recall that invariants and coinvariants of a $\La$-module have the same rank. For $i=0$, observe that we have, by \eqref{e:alg-arithm},
$$\rank\big((N^0_\infty)^\vee\big)_\Gamma=\rank\big((N^0_\infty)^\Gamma\big)^\vee\leq \rank\big(A_{p^\infty}(\Kpinf)^\Gamma\big)^\vee =\rank\big(A(K)[p^\infty]\big)^\vee=0\,.$$
From \eqref{e:alg-arithm} we get the exact sequence
$$0 \lr X_p(A/\Kpinf) \lr (N^1_\infty)^\vee \lr \M_\infty[p^\infty]^\vee.$$
Keeping the notations of \S\ref{ss:bsdinfty}, we have
$$\rank_{\ZZ_p}\big((\M_\infty)^\vee\big)^\Gamma =\rank_{\ZZ_p}\big(\oplus_{w|v}\big(B_w(k(w))[p^\infty]^\vee\big)^\Gamma\big)=0$$
where the first equality results from the facts that $\Phi_w$ is a finite group and $(T_w)[p^\infty]$ is trivial. Therefore we have
$$\rank_{\ZZ_p}\left((N^1_\infty)^\vee\right)^\Gamma = \rank_{\ZZ_p}X_p(A/\Kpinf)^\Gamma = \rank_{\ZZ_p} X_p(A/K) \,,$$
where the last equality is a consequence of the control theorem \cite[Theorem 4]{tan10a}. This yields \eqref{e:ranks}.
On the other hand, for the group $(N^2_\infty)^\Gamma$, we have
$$\rank_{\ZZ_p}(N^2_0)^\vee=\rank_{\ZZ_p}\Sel_{\ZZ_p}(A^t/K)\,,$$
by \eqref{e:N2compactSel}. On the other hand, we have
$$\rank_{\ZZ_p}\Sel_{\ZZ_p}(A^t/K)=\rank A^t(K)+\rank_{\ZZ_p}\Sha_{p^\infty}(A^t/K)^\vee=\rank_{\ZZ_p}X_p(A^t/K),$$
In particular,
$$\rank_{\ZZ_p}(N^2_0)^\vee = \rank_{\ZZ_p}X_p(A^t/K) = \rank_{\ZZ_p}X_p(A/K)\,,$$
where the second equality results from the existence of an isogeny between $A$ and $A^t$. Hence, from the short exact sequence \eqref{ses1} we deduce $\rank_{\ZZ_p}((N^2_\infty)^\Gamma)^\vee=0$.
\end{proof}

Let $\tau\in \coh^1_{\mathrm{\acute et}}(\FF_p,\ZZ_p)=\Hom_{cont}(\Gal(\bar{\FF}_p/\FF_p),\ZZ_p)$ be the element which sends the arithmetic Frobenius to 1. By \cite[Lemma 6.9]{KT03} the composed map
$$\cup'\colon N_0^1\lr M^1_{1,0}\buildrel{{\bf 1}}\over\lr M^1_{2,0}\lr N_0^2$$
coincides up to sign with $\tau\cup$, the cup product by the image of $\tau$ in $\coh^1_{\mathrm{fl}}(X, \ZZ_p)$. Using \cite[Proposition 6.5]{Mil86b}, we can prove the following result.

\begin{lemma}\label{cup}
The composed map
$$\begin{CD} \cup\colon N^1_0\buildrel{\beta}\over\lr (N^1_\infty)^\Gamma @>{g_{N^1_\infty}}>> (N^1_\infty)_\Gamma \buildrel{\gamma}\over\lr N^2_0 \end{CD}$$
coincides (up to the sign) with the map
$$\tau\cup=\cup'\colon N^1_0\, \lr N^2_0.$$
\end{lemma}

\subsubsection{The height pairing}\label{iwasawapairing} We denote 
$$\tilde{h}_{A/K}: A(K)\times A^t(K)\to \RR,$$ the N\'eron-Tate height pairing (see for example \cite{lan83}). Let $e_1,...e_r$ be elements of $A(K)$ which form a $\ZZ$-basis of $A(K)/A(K)_{tor}$ and let $e_1^*,...,e_r^*$ be elements of $A^t(K)$ which form a $\ZZ$-basis of $A^t(K)/A^t(K)_{tor}$. Then
$$Disc(\tilde h_{A/K}):=|\det(\tilde h_{A/K}(e_i,e_j^*)_{i,j})|\in\RR$$
is independent of the choices of basis and we call it the discriminant of the height pairing. It is known that $Disc(\tilde h_{A/K})\neq 0$. We write
$$Disc(h_{A/K}) = \log(p)^{-r} Disc(\tilde h_{A/K}),$$
with $r=\rank(A(K))$.

Consider the quotient category $(ab)/(fab)$, where $(ab)$ is the category of abelian groups and $(fab)$ the category of finite abelian groups. Let $\theta$ be the composed map in $(ab)/(fab)$ defined by:
\begin{equation} \label{e:relhp} \begin{CD}  A(K) \otimes \QQ_p/\ZZ_p @>{\alpha}>>  N^1_0 \buildrel{\beta}\over\lr (N^1_\infty)^\Gamma @>{g_{N^1_\infty}}>> (N^1_\infty)_\Gamma \\
@V{\theta}VV &&    @V{\gamma}VV\\
\Hom\big(A^t(K)/A^t_{p^\infty}(K), \QQ_p/\ZZ_p\big) @<{v^{-1}}<< \Hom(A^t(K), \QQ_p/\ZZ_p) @<{u}<< N^2_0
\end{CD}\end{equation}
where: \begin{enumerate}
\item the map $\alpha\colon A(K)\otimes \QQ_p/\ZZ_p \rightarrow N_0^1$ is the canonical morphism in $(ab)/(fab)$ coming from $A(K)\otimes \QQ_p/\ZZ_p \to \Sel_{p^\infty}(A/K)$, by \eqref{e:alg-arithm0};
\item the map $u\colon N^2_0\to \Hom(A^t(K), \QQ_p/\ZZ_p)$ is the map constructed by using the isomorphism \eqref{e:KT253} and the natural map $A^t(K)\otimes \ZZ_p\to \Sel_{\ZZ_p}(A^t/K)$;
\item the map $v$ is induced by the quotient map (which is an isomorphism, since we are in the quotient category $(ab)/(fab)$).
\end{enumerate}
Thanks to Lemma \ref{cup} and \cite[3.3.6.2 and \S6.8]{KT03}, $\theta$ coincides (up to sign) with the map induced by $h_{A/K}$ in $(ab)/(fab)$. In particular, since the N\'eron-Tate height pairing is non-degenerate, $\theta$ is a quasi-isomorphism (i.e., an isomorphism in the quotient category).

\subsubsection{Computation of the Euler characteristic} If $f$ is a quasi-isomorphism of abelian groups, we denote $char(f):=|\Ker(f)|/|\Coker(f)|$. Since this characteristic is multiplicative, \eqref{e:relhp} gives
$$char(\alpha)\cdot char(\beta)\cdot char(g_{N^1_\infty})\cdot char(\gamma)\cdot char(u)\cdot char(v)^{-1} = char(\theta) \equiv_p Disc(h_{A/K})$$
where $ \equiv_p$ means ``$\equiv$ mod $\ZZ_p^{\times}$''. If we assume that $A/K$ has semistable reduction and the Tate-Shafarevich group of $A/K$ is finite, then we have: \label{euler}
$$char(\alpha)=\frac{|A_{p^\infty}(K)|}{|\Sha_{p^\infty}(A/K)|\cdot |\M_0[p^\infty]|\cdot|N^0_0|}$$
by \eqref{e:alg-arithm0};
$$char(v)\equiv_p|A_{p^\infty}^t(K)|^{-1}$$
because $v\colon \Hom\big(A^t(K)/A^t_{p^\infty}(K), \QQ_p/\ZZ_p\big) \rightarrow \Hom(A^t(K), \QQ_p/\ZZ_p)$ is injective with cokernel $A^t_{p^\infty}(K)^\vee$;
$$char(\beta)=|(N^0_\infty)_\Gamma|$$
and
$$char(\gamma)=|(N^2_\infty)^\Gamma|^{-1}\,,$$
by \eqref{ses0} and \eqref{ses1}, using Lemma \ref{N02}; and $char(u)=1$ since the Tate-Shafarevich group is assumed to be finite.

The following result can be seen as a geometric analogue of the conjecture of Mazur-Tate-Teitelbaum (\cite{mtt86}):
\begin{mytheorem}\label{euchar} Assume that $A/K$ has semistable reduction. Then
\begin{equation}\label{e:rankss}
ord\big(\cL_{A/\Kpinf}\big)=ord_{s=1}\big(L_Z(A,s)\big)\geq\rank_\ZZ A(K)\,.
\end{equation}
If moreover $A/K$ verifies the Birch and Swinnerton-Dyer Conjecture, the inequality above becomes an equality and
\begin{equation}\label{e:euchar}
|L(\cL_{A/\Kpinf})|_p^{-1}\equiv  c_{BSD}\cdot |(N_\infty^2)_\Gamma| \mod \ZZ_p^\times \,,
\end{equation}
where $c_{BSD}$ is the leading coefficient at $s=1$ of $L_Z(A,s)$.
\end{mytheorem}
\begin{proof}
The last inequality in \eqref{e:rankss} has been proved in \cite[\S3.5]{KT03}. We show that the analytic rank is equal to the rank of our $p$-adic $L$ function $\cL_{A/\Kpinf}$. First, note that the operator $(\varphi_{i,0}\otimes\Fr_q)^\vee$ on $P^i_\infty$ is induced by the operator  $(p^{-1}F_{i,0}\otimes\Fr_q)^\vee$ on $(P^i_0[\frac{1}{p}]\otimes \QQ_{p,\infty})^\vee$. Moreover, using \eqref{exact1}, we observe that we have an injection of $\QQ_p[[\Gamma]]$-modules
$$P^i_\infty\otimes_{\La[\frac{1}{p}]}\QQ_p[[\Gamma]]\hookrightarrow (P^i_0[p^{-1}]\otimes \QQ_{p,\infty})^\vee \simeq\QQ_p[[\Gamma]]^{r_i}$$
with $r_i:=\dim_{\QQ_p}(P^i_0[\frac{1}{p}])$ and where the operator $(p^{-1}F_{i,0}\otimes\Fr_q)^\vee$ on the left-hand side corresponds to the operator $\Fr_q\cdot p^{-1}F_{i,0}$ on the right-hand side (as shown in Lemma \ref{induce}).
Also, Lemma \ref{linalg1} shows that $P^i_\infty$ is a free $\La[\frac{1}{p}]$-module of rank equal to the $\ZZ_p$-corank of $M^i_{2,0}$, which, by Lemma \ref{cofinite}, is precisely $r_i$.
Hence, the $p$-adic $L$-function $\cL_{A/\Kpinf}$ can be written
\begin{equation} \label{e:detfr}
\prod_{i=0}^2 {\det}_{\QQ_p[[\Gamma]]}\left(id-\Fr_q\cdot p^{-1}F_{i,0},P^i_\infty\otimes\QQ_p[[\Gamma]]\right)^{(-1)^{i+1}}=\prod_{i=0}^2(\prod_{j=1}^{r_i}(1-\alpha_{ij}\Fr_q))^{(-1)^{i+1}},
\end{equation}
where the $\alpha_{ij}$'s are the eigenvalues (in ${\bar \QQ}_p$) of $p^{-1}F_{i,0}$. In particular, $ord_{\bf 1}(\prod_{j=1}^{r_i}(1-\alpha_{ij}\Fr_q))$ is the number of $\alpha_{ij}$ equal to $1$ (note that $1-\lambda\Fr_q$ has order 0 if $\lambda\neq 1$ and order 1 else), that is, the multiplicity of the eigenvalue 1 of the operator $p^{-1}F_{i,0}$ and the assertion follows from \cite[3.5.2]{KT03}.

We proceed to the proof of the second assertion of the theorem.
By Lemma \ref{N02}, $(N^0_\infty)^\vee$ and $(N^2_\infty)^\vee$ have finite generalized $\Gamma$-characteristic.
For $(N^1_\infty)^\vee$, remark that under our assumption $\alpha$, $\beta$ and $v$ are quasi-isomorphisms while $u$ and $\gamma$ are isomorphisms. Therefore, the map $g_{N^1_\infty}$ is a quasi-isomorphism and so is its Pontryagin dual, $g_{(N^1_\infty)^\vee}$. By Theorem \ref{chi-Iwasawa-BSD} and Lemma \ref{N02}, we have
$$ord_{\bf 1}(\cL_{A/\Kpinf})=\rank_{\ZZ_p}X_p(A/K)=\rank_\ZZ A(K)\,,$$
since we have assumed that the Tate-Shafarevich group of $A/K$ is finite. The last equality $\rank_\ZZ A(K)=ord_{s=1}L_Z(A,s)$ follows from the main theorem of \cite{KT03}. The same theorem also proves that if $\alpha_v$ and $\mu_v$ are the Haar measure defined in \S\ref{ss:bsdinfty}, then
$$c_{BSD}=\frac{|\Sha(A/K)|\cdot Disc(h_{A/K})}{|A(K)_{tor}|\cdot|A^t(K)_{tor}|}\cdot\mu(Lie({\cal A})( \mathbb{A}_K)/Lie({\cal A})( K))^{-1}\cdot \prod_{v\in Z}\alpha_v(A(K_v))\,.$$
Replacing the values above and applying \eqref{e:tamagawa} and \eqref{e:tamagawa3} one gets
$$ c_{BSD} \equiv_p q^{-g(\deg(Z)+\kappa-1)-\delta}\cdot\frac{|(N^0_\infty)_\Gamma|\cdot char(g_{N^1_\infty})}{|N^0_0|\cdot|(N^2_\infty)^\Gamma|}\,. $$

On the other hand, by Theorem \ref{chi-Iwasawa-BSD},
$$|\pounds_{\bf 1}(\cL_{A/\Kpinf})|_p^{-1}=q^{-g(\deg(Z)+\kappa-1)-\delta}\cdot\frac{char(\Gamma,(N^1_\infty)^\vee)}{char(\Gamma,(N^0_\infty)^\vee)\,char(\Gamma,(N^2_\infty)^\vee)}$$
and so the result follows remembering \eqref{e:EulerMdual} and observing that if the $\La$-module $M^\Gamma$ has finite cardinality then $char(\Gamma,M)=|M_\Gamma|/|M^\Gamma|$ (note also that $(N^0_\infty)^\Gamma=N^0_0$).

\end{proof}

\end{document}